\documentclass[11pt]{article}
\usepackage{amsmath,amsthm,amsfonts,amssymb,amscd, amsxtra, color}
\usepackage[margin=2cm,nohead]{geometry}

\newcommand{\banacha}{\mathbb X}
\newcommand{\banachb}{\mathbb Y}
\newcommand{\banachc}{\mathbb Z}
\newcommand{\norm}[1]{\|#1\|}
\newcommand{\Norm}[1]{\left\|#1\right\|}

\DeclareMathOperator*{\argmin}{arg\,min}
\newtheorem{theorem}{Theorem}
\newtheorem{lemma}[theorem]{Lemma}

\newtheorem{corollary}[theorem]{Corollary}
\newtheorem{proposition}[theorem]{Proposition}

\newtheorem{remark}{Remark}

\begin{document} 
\title{ Inexact Newton's method to nonlinear functions with values in  a cone}

\author{ 
O.  P. Ferreira\thanks{IME/UFG,  CP-131, CEP 74001-970 - Goi\^ania, GO, Brazil (Email: {\tt
      orizon@ufg.br}). The author was supported in part by CAPES (Projeto  019/2011- Coopera\c c\~ao Internacional Brasil-China), CNPq Grants 471815/2012-8, 444134/2014-0 and  305158/2014-7, PRONEX--Optimization(FAPERJ/CNPq) and FAPEG/GO.} 
\and 
G. N. Silva \thanks{CCET/UFOB,  CEP 47808-021 - Barreiras, BA, Brazil (Email: {\tt  gilson.silva@ufob.edu.br}). The author was supported in part by CAPES .}    
}
\date{April 21, 2016}  
\maketitle
\begin{abstract}

The problem of finding a  solution of  nonlinear inclusion problems in Banach space  is considered in this paper.   Using convex optimization techniques  introduced by Robinson (Numer. Math., Vol. 19, 1972, pp. 341-347),   a robust convergence theorem  for inexact Newton's method is proved. As an application,   an  affine  invariant version of Kantorovich's theorem and Smale's  $\alpha$-theorem for inexact Newton's method is obtained. \\

\noindent
{\bf Keywords:} Inclusion problems, inexact Newton's method, majorant condition, semi-local convergence.

\end{abstract}
\section{Introduction}\label{sec:int}
In this paper we study the inexact Newton's method for solving the nonlinear  inclusion problem
\begin{equation} \label{eq:ipi}
  F(x) \in C, 
\end{equation}
where $F:{\Omega}\to \banachb$ is a  nonlinear continuously   differentiable function,   $\banacha$  and $\banachb$ are Banach spaces,   $\banacha$ is reflexive, $\Omega\subseteq \banacha$ an open set and  $C\subset \banachb $ a nonempty closed convex cone. The idea of solving a nonlinear  inclusion problems of the form \eqref{eq:ipi},   plays a huge role in classical analysis and its applications. For instance,  the  special case in which $C$ is the degenerate cone $\{0\}\subset \banachb$,  the  inclusion  problem  in  \eqref{eq:ipi} corresponds to a nonlinear equation.   In the case for which $\banacha=\mathbb{R}^{n}$, $\banachb=\mathbb{R}^{p+q}$ and $C=\mathbb{R}^{p}_{-}\times \{0\}$ is the product of the nonpositive orthant in $\mathbb{R}^{p}$ with the origin in  $\mathbb{R}^{q}$,  the  inclusion  problem  in  \eqref{eq:ipi} corresponds to a nonlinear system of  $p$  inequalities and  $q$ equalities,  for example see \cite{BlumSmale1998}, \cite{Dennis1996}, \cite{Daniel1973}, \cite{Deuflhard2004}, \cite{DontchevRockafellar2009}, \cite{HeJie2005}, \cite{Krantz2013}, \cite{LiNg2012} and \cite{Psenicyi1970}. 

In  order to solving \eqref{eq:ipi},     in  \cite{Robinson1972_2}  the following   Newton-type iterative method was proposed:
\begin{equation}\label{eq:nmc}
x_{k+1}= x_k + d_k, \qquad \quad d_k\in  \argmin_{d\in \banacha}\left\{\|d\|~:~ F(x_k)+F'(x_k)d \in C\right\},\qquad \quad  k=0,1, \ldots .
\end{equation}

In general, this algorithm may fail to converge and may even fail to be well defined. To ensure that the method is well defined and converges to a solution of the  nonlinear  inclusion,   S. M. Robinson,  made two important assumptions: 
\begin{itemize}
\item[{\bf H1.}] There exists   $x_0\in \banacha$ such that  $\mbox{rge}~ T_{x_0}= \banachb,$ where      $T_{x_0}: \banacha \rightrightarrows  \banachb$  is the convex process given by 
$$
T_{x_0} d:= F'(x_0)d -C, \qquad  \quad ~ d\in \banacha, 
$$
 and $\mbox{rge}~ T_{x_0}=\{ y\in \banachb ~:~ y\in T_{x_0}(x) ~ \mbox{for some}~ x \in \banacha\}$, see \cite{DontchevRockafellar2009} for additional details. 
\item[{\bf  H2.}]    $F'$ is Lipschitz continuous  with  modulo  $L$, i.e.,  $\left\| F'(x) - F'(y)\right\| \leq L \left\| x-y\right\|$,  for all $x, y,\in \banacha $. 
\end{itemize}
Under these assumptions, it was proved in \cite{Robinson1972_2},  that the  sequence $\{x_k\}$ generated by  (\ref{eq:nmc}) is well defined and converges to  $x_*$ satisfying $F(x_*) \in C$,  provided that following convergence criterion is satisfied:
$$
\|x_1-x_0\| \leq \frac{1}{2L\|T_{x_0}^{-1}\|}.
$$
The first  affine invariant version of this result was  presented in \cite{LiNg2012}.  In \cite{LiNg2013} they introduced the notion of the weak-Robinson condition for convex processes and presented an extension of the results of \cite{LiNg2012} under an $L$-average Lipschitz condition. As applications,  two special cases  were provided, namely, the convergence result of the method under Lipschitz's condition and Smale's condition.  In \cite{Ferreira2015}, under an affine majorant condition,  a robust analysis  of this  method were established. As in \cite{LiNg2012}, the analysis under Lipschitz's condition and Smale's condition are also obtained as special case, see also \cite{Alvarez2008}, \cite{DedieuPrioureMalajovich2003}. 

The inexact Newton method, for solving nonlinear equation $F(x)=0$,  was introduced in \cite{DemboEisenstat1982}  for denoting any method which, given an initial point $x_0$, generates the sequence $\{ x_k\}$ as follows:
\begin{equation} \label{eq:inm}
\|F(x_k) + F'(x_k)(x_{k+1}-x_k)\| \leq \eta_k\|F(x_k)\|, \qquad k=0,1,\ldots, 
\end{equation}
and $\{\eta_k\}$ is a sequence of forcing terms such that $0\leq \eta_k <1$;  for others variants of  this method see \cite{ArgyrosHilout2010},  \cite{DontchevRockafellar2013}, \cite{FerreiraSvaiter2012}.   In \cite{DemboEisenstat1982}  was  proved,  under suitable assumptions,  that $\{x_k\}$  is convergent to a solution with  super-linear rate.  In  \cite{Kelley2003}  numerical issues about this method are discussed. In the present paper, we extend the inexact Newton's method  (\ref{eq:inm}), for solving nonlinear inclusion,  as any method which, given an initial point $x_0,$ generates a sequence $\{x_k\}$ as follows:  
\begin{equation}\label{eq:inmc2}
 x_{k+1}={x_k}+d_k, \qquad d_k \in  \argmin_{d\in \banacha}\left\{\|d\| ~:  ~ F(x_k)+F'(x_k)d +r_k \in C \right\}, 
\end{equation}
\begin{equation}\label{eq:inmc3}
  \max_{w\in \{-r_k,~ r_k  \}} \left\|T_{x_0}^{-1}w\right\|\leq \theta \left\|T_{x_0}^{-1}[-F(x_k)]\right\|, 
\end{equation}
for  $  k=0,1, \ldots  $,    $0\leq \theta < 1$ is   a fixed suitable tolerance, and   $T_{x_0}^{-1}(y):=\left\{d\in \banacha ~:~ F'(x_0)d-y \in C\right\} $, for $y\in \banachb$.  We point out that, if $\theta =0$ then \eqref{eq:inmc2}-\eqref{eq:inmc3} reduces to extended Newton method (\ref{eq:nmc}) for solving \eqref{eq:ipi} and, in the case,   $C=\{0\}$ it reduces to affine invariant version of (\ref{eq:inm}), which was also studied in \cite{FerreiraSvaiter2012}.  

It is worth noting that  (\ref{eq:ipi}) is a particular instance  of the following  generalized equation
\begin{equation}\label{eq:geint}
F(x)+T(x)\ni 0, 
\end{equation}
when $T(x)\equiv -C$ and   $T : \banacha \rightrightarrows \banachb$ is a set valued mapping.  In  \cite{DontchevRockafellar2013} (see also \cite{Dontchev2015}), they  proposed the following    Newton-type method for solving \eqref{eq:geint}:
\begin{equation}\label{eq:gein}
(F(x_k) +F'(x_k)(x_{k+1}-x_k)+T(x_{k+1}))\cap R_{k}(x_k, x_{k+1}) \neq \varnothing, \qquad \;\;k=0,1, \ldots , 
\end{equation} 
where  $R_{k}: \banacha \times \banacha \rightrightarrows \banachb$ is a sequence of set-value mappings with closed graphs. Note that,  in the case,  when $C(x)\equiv 0$, $\theta\equiv \eta_k$ and 
$$
R_k(x_k,x_{k+1})\equiv B_{\eta_k\|F(x_k)\|}(0),
$$ 
the iteration (\ref{eq:gein})   reduces to (\ref{eq:inm}). We also remark that,  in the particular  case $T(x)\equiv-C$, the iteration (\ref{eq:gein}) has  \eqref{eq:inmc2}-\eqref{eq:inmc3}  as  a minimal norm  affine invariant version. Therefore, in some sense, our method is a particular  instance of \cite{DontchevRockafellar2013}. However, the  analysis presented in   \cite{DontchevRockafellar2013} is local, i.e.,  it is made assumption at a solution,    while in our analysis we will not assume existence of solution.  In fact,   our aim is  to prove a  robust  Kantorovich's Theorem for  \eqref{eq:inmc2}-\eqref{eq:inmc3},  under assumption  {\bf H1} and  an  affine invariant   majorant condition generalizing   {\bf H2}, which   in particular, prove  existence of solution for (\ref{eq:ipi}).  Moreover, the analysis presented, shows that the robust analysis of the inexact  Newton's method for solving nonlinear inclusion problems, under affine Lipschitz-like and affine Smale's conditions,  can be obtained as a special case of the general theory. Besides, for the degenerate cone, which the nonlinear inclusion becomes a nonlinear equation, our analysis retrieves the classical results on semi-local analysis of  inexact Newton's method; \cite{FerreiraSvaiter2012}.  Up to our knowledge, this is the first time that the inexact Newton method to solving  cone inclusion problems with a relative error tolerance  is analyzed.

The organization of the paper is as follows. In Section \ref{sec:int.1},  some notations and  basic results  used in the paper are presented. In Section \ref{lkant}, the main results are stated and  in   Section \ref{sec:PR} some properties of the majorant function are established and the main relationships  between the majorant function and the nonlinear operator used in the paper are presented. In Section \ref{sec:inpso},  the main results are proved and the applications of this results are given in Section \ref{sec:scinmer}. Some final remarks are made in Section~\ref{sec:fr}.

\subsection{Notation and auxiliary results} \label{sec:int.1}
 Let $\banacha$  be a  Banach space. The {\it open} and {\it closed ball} at $x$ with radius $\delta>0$ are denoted, respectively, by $ B(x,\delta) := \{ y\in \banacha ~:~ \|x-y\|<\delta \}$ and $B[x,\delta] := \{ y\in \banacha  ~:~\|x-y\|\leqslant \delta
\}.$ A  set valued mapping $T: \banacha \rightrightarrows  \banachb $ is called {\it sublinear} or   {\it convex process}    when its graph is a convex cone, i.e., 
\begin{equation} \label{eq:dsblm}
0\in T(0), \qquad T(\lambda x)=\lambda T(x),   \qquad  \lambda>0, \qquad T(x+x') \supseteq T(x) + T(x'),  ~\qquad  ~ x, x'\in \banacha,
\end{equation}
 (sublinear mapping has been extensively studied in \cite{DontchevRockafellar2009},  \cite{Robinson1972_1},  \cite{Rockafellar1967} and  \cite{Rockafellar1970}).  The {\it domain} and  {\it range}   of  a sublinear mapping $T$ are defined, respectively, by $  \mbox{dom\,}T:=\{d\in \banacha~:~Td \neq \varnothing \}, $ and   $ \mbox{rge\,}T:=\{y\in \banachb ~:~ y\in T(x) ~\mbox{for some} ~x\in\banacha \}.$ The {\it norm} (or inner norm as is called in \cite{DontchevRockafellar2009})   of  a sublinear mapping $T$ is defined  by
 \begin{equation} \label{eq;dn}
 \|T\|:=\sup \; \{ \|T d\|~: ~d\in  \mbox{dom\,}T, \; \|d\| \leqslant 1 \},
 \end{equation}
where $ \|T d\|:=\inf \{\|v\|~: ~v\in T d  \}$ for  $Td \neq \varnothing $. We use the convention $\|Td \|=+\infty$ for $Td = \varnothing $,   it will be also convenient to use the convention  $Td+ \varnothing=\varnothing$ for all $d\in \banacha$. Let $S, T: \banacha \rightrightarrows  \banachb $ and $U: \banachb \rightrightarrows  \banachc$  be sublinear mappings.  The   scalar {\it multiplication}, {\it addition}  and {\it composition} of sublinear mappings are sublinear mappings defined, respectively,  by
$(\alpha S)(x):=\alpha S(x),$  $(S+T)(x):= S(x) + T(x),$ and  $UT(x):=\cup \left\{ U(y)~: ~ y\in T(x) \right \},$ for all $x\in \banacha$ and  $\alpha>0$  and the following norm properties there hold $ \|\alpha S\|=|\alpha| \|S\|,$ $\|S+T\|\leqslant \|S\|+\|T\| $ and $\|UT\|\leqslant \|U\| \|T\|.$
\begin{remark} \label{r:pn}
Note that definition of the norm in \eqref{eq;dn} implies that  if $ \mbox{dom\,}T=\banacha$ and $A$ is a linear mapping from $\banachc$ to $\banacha$ then 
$\|T(-A)\|=\|TA\|$.
\end{remark}
Let $\Omega\subseteq \banacha$ be an open set and $F:{\Omega}\to \banachb$  a continuously Fr\'echet differentiable function.   The   linear map $F'(x):\banacha \to \banachb$ denotes the  Fr\'echet derivative of $F:{\Omega}\to \banachb$ at $x\in \Omega$.  Let  $C \subset \banachb $ be  a  nonempty closed convex cone,  $z\in\Omega$ and   $T_z: \banacha \rightrightarrows  \banachb $ a  mapping  defined  as
\begin{equation} \label{ro2}
T_{z}d:=F'(z)d-C.
\end{equation}
It is well known that the mappings  $T_{z}$  and  $T^{-1}_{z}$ are sublinear  with closed graph, $ \mbox{dom\,}T_z=X$, $\|T_z\|<+\infty$ and, moreover,  $\mbox{rge\,}T_z=Y$ if and only if $ \|T^{-1}_{z}\|<+\infty$ (see Lemma 3 above and  Corollary 4A.7, Corollary 5C.2 and Example 5C.4 of \cite{DontchevRockafellar2009} ). Note that 
\begin{equation} \label{ro3}
 T^{-1}_zy:=\{d\in \banacha~:~ F'(z)d-y\in C\}, \qquad  \quad  z\in  \Omega, ~ ~ y\in \banachb.
\end{equation}
\begin{lemma} \label{l:incltr}
There holds $ T_{z}^{-1}F'(v)T_{v}^{-1}w \subseteq T_{z}^{-1}w,$  for all $ v, z\in  \Omega, ~ ~ w\in \banachb.  $
As a consequence, 
$$
 \left\|T_{z}^{-1}\left[ F'(y)-F'(x)\right]\right\|\leq \left\|  T_{z}^{-1}F'(v)T_{v}^{-1}\left[ F'(y)-F'(x)\right] \right\|, \qquad  v, x, y, z\in  \Omega.
$$
\end{lemma}
\begin{proof}
See  \cite{Ferreira2015}.
\end{proof}
\section{Inexact Newton's  method } \label{lkant}
Our goal is to state and prove a robust semi-local affine invariant theorem for inexact Newton's method to solve nonlinear inclusion  of the form \eqref{eq:ipi},   for state this theorem we need some definitions. 

Let $\banacha$, $\banachb$ be Banach spaces,  $\banacha$ reflexive, $\Omega\subseteq \banacha$ an open set,    $F:{\Omega}\to \banachb$  a continuously Fr\'echet differentiable function. The  function $F$ satisfies the {\it Robinson's  Condition} at $x_0\in \Omega$ if 
$$
\mbox{rge\,}T_{x_0}=\banachb, 
$$
where $T_{x_0}: \banacha \rightrightarrows  \banachb $  is a sublinear   mapping  as defined in \eqref{ro2}. Let   $R>0$ a scalar constant. A  continuously differentiable  function $f:[0,R)\to \mathbb{R}$ is a {\it majorant function}   at a point  $x_0 \in \Omega$ for     $F$     if 
  \begin{equation}\label{eq:MCAI}
   B(x_0,R)\subseteq \Omega, \qquad \qquad  \left\|T_{x_0}^{-1}\left[F'(y)-F'(x)\right]\right\| \leqslant  f'(\|x-x_0\| + \|y-x\|)-f'(\|x-x_0\|),
  \end{equation}
  for all $x,y\in B(x_0,R)$ such that  $\|x-x_0\|+\|y-x\|< R$ and satisfies the  following conditions:
  \begin{itemize}
  \item[{\bf h1)}]  $f(0)>0$,   $f'(0)=-1$;
  \item[{\bf h2)}]  $f'$ is convex and strictly increasing;
  \item[{\bf h3)}]  $f(t)=0$ for some $t\in (0,R)$.
  \end{itemize}
We also need of the following  condition on the majorant condition $f$  which will be considered to hold
only when explicitly stated.
  \begin{itemize}
  \item[{\bf h4)}]  $f(t)<0$ for some $t\in (0,R)$.
  \end{itemize}
Note that the condition {\bf h4} implies the condition {\bf h3}.  

The  sequence $\{z_k\}$ generated by {\it inexact Newton's method} for solving   the inclusion $ F(x)\in C$
   with starting point  $z_0$ and residual relative error tolerance $\theta$  is defined by:  $ z_{k+1}:={z_k}+d_k,$
 $$
 d_k \in  \argmin_{d\in \banacha}\left\{\|d\| ~:  ~ F(z_k)+F'(z_k)d +r_k \in C \right\},  \qquad  \max_{w\in \{-r_k,~ r_k  \}} \left\|T_{x_0}^{-1}w\right\|\leq \theta \left\|T_{x_0}^{-1}[-F(z_k)]\right\|, 
 $$
 for  $k=0,1, \ldots$.  The statement of the our  main theorem is:
  \begin{theorem}\label{th:knt1}
Let   $C\subset \banachb $ a nonempty closed convex cone, $R>0$.   Suppose that $x_0 \in \Omega$, $F$ satisfies  the Robinson's condition at  $x_0$,    $f$ is a majorant function for $F$ at  $x_0$ and
\begin{equation} \label{KH}
   \left \|T_{x_0}^{-1}[-F(x_0)]\right\|\leqslant f(0)\,.
\end{equation}
Let $ \beta:=\sup\{ -f(t) ~:~ t\in[0,R)  \}$. Take $0\leq \rho<\beta/2$ and define the constants
  \begin{equation} \label{eq:dktt}
  \kappa_\rho:=\sup_{\rho <t<R}\frac{-(f(t)+2\rho)}{|f'(\rho)|\,(t-\rho)},\qquad  \lambda_\rho:=\sup \{t\in [\rho,R): \kappa_\rho+f'(t)<0\},
  \qquad  \tilde \theta_\rho:=\frac{\kappa_\rho}{2-\kappa_\rho}.
  \end{equation}
  Then for any $\theta\in [0,\tilde \theta_\rho]$ and $z_0\in B(x_0,\rho)$,  the sequence $\{z_k\}$,  is well defined, for any particular choice of each $d_k$, 
	\begin{align}\label{FC}
	\|T_{z_0}^{-1}[-F(z_k)]\| \leq \left(\frac{1+\theta^2}{2}\right)^k\left[f(0) +2\rho\right],
	\end{align}
$\{z_k\}$ is contained in $B(z_0, \lambda_\rho )$ and converges to a point $x_*\in B[x_0, \lambda_\rho]$  such that  $ F(x_*)\in C$. Moreover, if
  \begin{itemize}
  \item[{\bf h5)}] $\lambda_\rho<R-\rho$,
  \end{itemize}
 then the sequence $\{z_k \}$ satisfies, for $k= 0,1, \ldots \,$, 
\begin{equation} \label{eq:slc}
\|z_k-z_{k+1}\|\leq \frac{1+\theta}{1-\theta}\left[ \frac{1+\theta}{2}\frac{D^{-}f'(\lambda_\rho +\rho)}{|f'(\lambda_\rho+\rho)|}\|z_k-z_{k-1}\|
+\theta\,\frac{2|f'(\rho)|+f'(\lambda_\rho+\rho)}{|f'(\lambda_\rho+\rho)|}\right]\|z_k-z_{k-1}\| .
\end{equation}
  If, additionally, $0\leq \theta <[-2(\kappa_{\rho}+1)+\sqrt{4(\kappa_{\rho}+1)^2+\kappa_{\rho}(4+\kappa_{\rho})}]\big/[4+\kappa_{\rho}]$ then  $\{z_k \}$ converges $Q$-linearly as  follows
 \begin{equation} \label{eq:lc}
	\limsup_{k \to \infty}\frac{\norm{x_*-z_{k+1}}}{\norm{x_*-z_k}}\leq \frac{1+\theta}{1-\theta}\left[ \frac{1+\theta}{2} +\frac{2\theta}{\kappa_{\rho}}\right], \qquad k= 0,1, \ldots \,.
\end{equation}
\end{theorem}

\begin{remark}
In Theorem~\ref{th:knt1} if $\theta=0$ we obtain the  exact Newton method as in \cite{Ferreira2015} and
its convergence properties. Now, taking $\theta=\theta_k$  in each iteration
and letting $\theta_k$ goes to zero as $k$ goes to infinity, inequality
\eqref{eq:slc}  implies that the  sequence $\{z_k\}$ converges to the solution of \eqref{eq:ipi} with asymptotic  superlinear rate. If $C=\{0\}$ we obtain the inexact 
Newton method as in \cite{FerreiraSvaiter2012} and its convergence properties are similar.
\end{remark}

Henceforth we assume that the assumption on  Theorem~\ref{th:knt1} holds, except   {\bf h5} which will be considered to hold
only when explicitly stated.

\subsection{Preliminary results} \label{sec:PR}
We will first prove Theorem~\ref{th:knt1} for the case $\rho=0$ and $z_0=x_0$. 
In order to  simplify the notation in the case $\rho=0$, we will use
$\kappa$, $\lambda$ and $\theta$ instead of $\kappa_0$, $\lambda_0$ and $\tilde\theta_0$ respectively:
\begin{equation} \label{eq:dktt.0}
  \kappa:=\sup_{0<t<R}\frac{-f(t)}{t},\qquad   \quad  \lambda:=\sup \{t\in [0,R): \kappa+f'(t)<0\},  \qquad \quad   \tilde \theta:=\frac{\kappa}{2-\kappa}.
\end{equation}
    
\subsubsection{The majorant function}
In this section we will prove the main results about the majorant function.
Define 
$$ 
 t_*:=\min f^{-1}(\{0\}),   \qquad  \qquad  \bar{t}:=\sup \left\{t\in [0,R): f'(t)<0 \right\}.
$$
Then we have the following remark about the above constants which was prove in  \cite[Proposition 2.4]{FerreiraSvaiter2012}:
\begin{remark} \label{pr:new}
 For $\kappa,\lambda,\theta$ as in \eqref{eq:dktt.0} it holds that $0<\kappa<1$, $0<\theta<1$ and  $t_*<\lambda\leq\bar t.$ Moreover, $f'(t)+\kappa<0$, for $t\in[0,\lambda)$ and $\inf_{0\leq t<R} ( f(t)+\kappa t)=\lim_{t\to\lambda_{-}} ( f(t)+\kappa t)=0$.
\end{remark}
The following proposition was proved in     \cite[Propositions 2.3 and 5.2]{FerreiraSvaiter2012} and  \cite[Proposition 3]{FerreiraSvaiter2009}.
\begin{proposition} \label{pr:maj.f}
  The majorant function $f$  has a smallest root $t_*\in  (0,R)$, is strictly convex and $ f(t)>0,$ $f'(t)<0$ and $ t<t-f(t)/f'(t)< t_*, $ for all $ t\in [0,t^*) .$
Moreover, $f'(t_*)\leqslant 0$ and $   f'(t_*)<0$  if,  and only if,   there exists $t\in (t_*,R)$ such that $f(t)< 0$.  If, additionally, $f$ satisfies  {\bf h4} then $f'(t)<0$ for any $t\in [0,\bar t)$,  $0< t_* < \bar t\leq R$,  $ \beta=-\lim_{t\to \bar t_{-}} f(t),$ $0< \beta <\bar t$ and if  $0\leq \rho<\beta/2$  then $\rho<\bar t/2 <\bar t$ and  $f'(\rho)<0$.
\end{proposition}
Take $0\leq \theta$ and $0\leq \varepsilon$. We will need of  the  following auxiliary mapping, which is associated to the inexact  newton iteration applied to the majorant function,     $n_{\theta}: [0,\bar t\,)\times [0,\infty) \to \mathbb{R}\times
\mathbb{R}$,
\begin{equation} \label{eq:nftheta}
  \displaystyle n_{\theta}(t,\varepsilon):
  =\left(t-(1+\theta)\frac{f(t)+\varepsilon}{f'(t)},\;\varepsilon+
 2\theta(f(t)+\varepsilon)\right), 
\end{equation}
The following auxiliary set will be  important for establishes the convergence of the inexact newton sequence  associated to the majorant function 
\begin{equation} \label{eq:dt} 
{\mathcal{A}}:=\left\{(t,\varepsilon)\in\mathbb{R}\times\mathbb{R}~:~0\leq t<\lambda, ~ 0\leq\varepsilon\leq \kappa t, ~ 0<f(t)+\varepsilon \right\}.
\end{equation}
The following lemma was proved in   \cite[Lemma 4.2]{FerreiraSvaiter2012}.
\begin{lemma} \label{lm:id.t}
  If $0\leq \theta\leq \tilde \theta$, $(t,\varepsilon)\in {\mathcal{A}}$   and $(t_+,\varepsilon_+):=n_{\theta}(t,\varepsilon)$, that is, $t_+:=t-(1+\theta)(f(t)+\varepsilon)/f'(t)$ and $\varepsilon_+:=\varepsilon+2\theta (f(t)+\varepsilon)$,  then $n_\theta(t,\varepsilon)\in {\mathcal{A}}$,   $t<t_+$ and $ \varepsilon \leq \varepsilon_+$. Moreover, 
  $f(t_+)+\varepsilon_+<[(1+\theta^2)/2](f(t)+\varepsilon).$
\end{lemma}
Define the   {\it linearization error} of  the majorant function  associated to  $F$ as follows
\begin{equation}\label{eq:def.ers}
        e_{f}(v,t):=f(v)-[f(t) +f'(t)(v-t)],   \qquad
     t, s\in [0,R).
\end{equation}
 We will need the following result about the linearization error, for  proving it  see   \cite[Lemma 3.3 ]{FerreiraSvaiter2012}.  
\begin{lemma} \label{l:errormon}
If  $0\leq b\leq t$, $0\leq a\leq s$ and $t+s<R$, then    there holds:
$$
 e_f(a+b,b) \leq   \max \left \{e_f(t+s,t),   ~  \frac{1}{2} \frac{f'(t+s)-f'(t)}{s}\;a^2\right\},  \qquad s\neq 0.
$$
\end{lemma}
\subsubsection{Relationships between the majorant and  nonlinear functions} 
In this section, we will present the main relationships between the majorant function $f$ and the nonlinear function $F$ that we need for proving Theorem~\ref{th:knt1}.   Note that  Robinson's condition, namely,  $\mbox{rge\,}T_{x_0}=\banachb$  implies  that $\mbox{dom\,}T_{x_0}^{-1}=\banachb$.
\begin{proposition} \label{wdns}
  If \,\,$\| x-x_0\|\leqslant t< \bar{t}$ then   $\mbox{dom\,}[T_{x}^{-1}F'(x_0)]=\banacha$ and there holds
$
\left\|T_{x}^{-1}F'(x_0)\right\|\leqslant  -1/f'(t). 
$
As a consequence, $\mbox{rge\,}T_{x}=\banachb$.
\end{proposition}
\begin{proof}
See   \cite[Proposition 12]{Ferreira2015}.
\end{proof}
Newton's iteration  at a point $x\in \Omega$ happens to be a solution  of the linearization
of the inclusion $F(y)\in C$ at such a point, namely, a solution of the linear  inclusion $F(x)+F'(x)(x-y)\in C$.  Thus, we study the  linearization error of $F$   at a point in $\Omega$
\begin{equation}\label{eq:def.er}
  E_{F}(y,x):= F(y)-\left[ F(x)+F'(x)(y-x)\right],\qquad y,\, x\in \Omega.
\end{equation}
We will bound this error by \eqref{eq:def.ers},  the error in the linearization on the majorant function associated to  $F$.
\begin{lemma} 
 \label{pr:taylor}
  If $x,y\in \banacha$ and $\norm{x-x_0}+\norm{y-x} < R$ then
  $
  \|T_{x_0}^{-1}E_F(y,x)\|\leq
  e_f(\norm{x-x_0}+\norm{y-x},\norm{x-x_0}).
$
\end{lemma}
\begin{proof}
As  $x, y\in B(x_0,R)$  and the ball is convex $ x+\tau(y-x)\in B(x_0,R),$
  for all $\tau\in [0,1]$.   Since, by assumption,  $\mbox{rge\,}T_{x_0}=\banachb$ we obtain   that $\mbox{dom\,}T_{x_0}^{-1}=\banachb$. Thus, using that $F'(z)$ is a linear mapping for each $z\in \banacha$,  we conclude 
$$
 \left\|T_{x_0}^{-1}\left([F'(x+\tau(y-x))-F'(x)](y-x)\right) \right\| \leq \left\|T_{x_0}^{-1}[F'(x+\tau(y-x))-F'(x)]\right\|\left\|y-x \right\|, 
$$
for all $\tau\in [0,1]$. Hence, as $f$ is a majorant function for $F$ at  $x_0$,  using  \eqref{eq:MCAI} and last inequality we have 
$$
  \left\|T_{x_0}^{-1}\left([F'(x+\tau(y-x))-F'(x)](y-x)\right) \right\|\leqslant \left[f'\left(\|x-x_0\|+\tau\left\|y-x\right\|\right)-f'\left(\|x-x_0\|\right)\right]\|y-x\|,
  $$
for all $\tau\in [0,1]$. Thus, since $\mbox{dom\,}T_{x_0}^{-1}=\banachb$,   we apply   Lemma 2.1 of  \cite{LiNg2012} with  $U=T_{x_0}^{-1}$ and the functions $G(\tau)$ and $g(\tau)$ equals to the expressions in the last inequality,   in  parentheses on  the left hand side  and  on the right hand side,  respectively, obtaining   
\begin{multline*}
    \left\|  T_{x_0}^{-1}\int_0 ^1 [F'(x+\tau(y-x))-F'(x)](y-x)\; d\tau \right\| \\ \leqslant \int_0 ^1   \left[f'\left(\|x-x_0\|+\tau\left\|y-x\right\|\right)-f'\left(\|x-x_0\|\right)  \right]\|y-x\|\;d\tau, 
\end{multline*}
	which, after performing the integration of the right hand side,  taking into account the definition of  $e_f(v, t)$ in \eqref{eq:def.ers}  and that \eqref{eq:def.er} is equivalent to 
   \[
   E_{F}(y,x)=\int_0 ^1 [F'(x+\tau(y-x))-F'(x)](y-x)\; d\tau,
   \]
  yields the desired inequality.
\end{proof}
\begin{lemma}  \label{pr:taylor2}
  If $x,y\in \banacha$ and $\norm{x-x_0}+\norm{y-x} < R$ then
  $
  \|T_{x_0}^{-1}[-E_F(y,x)]\|\leq e_f(\norm{x-x_0}+\norm{y-x},\norm{x-x_0}).
$
\end{lemma}
\begin{proof}
To prove this lemma we follow the same arguments used in the proof of  Lemma~\ref{pr:taylor},  by taking into account Remark~\ref{r:pn}.
\end{proof}
\begin{corollary} \label{cr:taylor}
If $x, y\in  \banacha$,  $\norm{x-x_0}\leq t$,  $\norm{y-x}\leq s$ and $s+t<R$ then
\begin{align*}
\max \left\{\left\|T_{x_0}^{-1}\left[-E_F(y,x)\right]\right\|,  ~ \|T_{x_0}^{-1}E_F(y,x)\| \right\}&\leq \max \left \{ e_f(t+s,t),~\frac{1}{2} \frac{f'(s+t)-f'(t)}{s}\norm{y-x}^{2}\right\}, \quad  s\neq 0. 
\end{align*}
\end{corollary}
\begin{proof}
The results follows by direct combination of the Lemmas \ref{pr:taylor}, \ref{pr:taylor2} and \ref{l:errormon} by taking $b=\norm{x-x_0}$ and $a=\norm{y-x}$.
\end{proof}
\begin{lemma} \label{l:bd}
If $x\in \banacha$ and  $\norm{x-x_0}\leq t< R$ then $\|T_{x_0}^{-1}F'(x)\|\leq 2+f'(t).$
\end{lemma}
\begin{proof}
First of all, we use Definition of sublinear mapping in (\ref{eq:dsblm}) to  obtain 
$$
T_{x_0}^{-1}F'(x) \supseteq T_{x_0}^{-1}[F'(x)-F'(x_0)] + T_{x_0}^{-1}F'(x_0).
$$
Hence, taking into  account properties of the norm,  we conclude from  above inclusion   that
$$
\Norm{T_{x_0}^{-1}F'(x)} \leq \Norm{T_{x_0}^{-1}[F'(x)-F'(x_0)]} + \Norm{T_{x_0}^{-1}F'(x_0)}.
$$
Since $T_{x_0}^{-1}F'(x_0)\supseteq F'(x_0)^{-1}F'(x_0)$ we have $\| \Norm{T_{x_0}^{-1}F'(x_0)}\|\leq 1$. Thus,  using  assumption (\ref{eq:MCAI}), the last inequality becomes
$$
\Norm{T_{x_0}^{-1}F'(x)} \leq f'(\norm{x-x_0})-f'(0)+1.
$$
Therefore,   assumptions {\bf h1}, {\bf h2}  and the last inequality imply the statement of the lemma.
\end{proof}
The next  result will be used to show that inexact Newton's method is robust with respect to the initial iterate, its prove can be found in   \cite[Proposition~16]{Ferreira2015}.
\begin{proposition} \label{pr:new.01}
If  $y\in B(x_0,R)$ then  $\Norm{T_{x_0}^{-1}[-F(y)]}  \leq f(\Norm{y-x_0})+2\Norm{y-x_0}.$
\end{proposition}
\section{Convergence analysis of the inexact Newton Method}\label{sec:inpso}
In this  section we will prove Theorem~\ref{th:knt1}. Before proving  Theorem~\ref{th:knt1},  we need to  study the  inexact Newton's iteration,  associated to the function  $F$, and prove Theorem~\ref{th:knt1} for the case $\rho=0$ and $z_0=x_0$. 
\subsection{The inexact Newton iteration}
The outcome of an inexact Newton iteration  is any point satisfying some error tolerance. Hence, instead of a mapping for inexact Newton iteration, we shall deal with a \emph{family} of mappings, describing all possible inexact iterations. Before defining  the inexact Newton iteration mapping,  we need to define the   {\it  inexact Newton's step mapping,  $D_{F,C,\theta}: B(x_0,\bar{t})  \rightrightarrows  \banacha$, 
\begin{equation} \label{def:dfct}
D_{F,C,\theta}(x):=\argmin_{d\in \banacha} \left\{ \|d\| ~: ~   F(x)+F'(x)d +r\in C\right\}; \quad \quad \max_{w\in \{-r,~ r  \}} \left\|T_{x_0}^{-1}w\right\|\leq \theta \left\|T_{x_0}^{-1}[-F(x)]\right\|, 
\end{equation}
associated to $F$,  $C$}  and $\theta$.   Since $\banacha$ is reflexive, second part of Proposition \ref{wdns} guarantees, in particular, that exact Newton's step $D_{F,C, 0}(x)$ is nonempty, for each  $x\in B(x_0,\bar{t})$.  Since $D_{F,C, 0}(x) \subseteq D_{F,C, \theta}(x)$, we conclude $D_{F,C, \theta}(x) \neq \varnothing$, for  $x\in B(x_0,\bar{t})$. Therefore,  for $0\leq \theta\leq \tilde \theta$, we can define  $ \mathcal{N}_\theta$  the {\it family of  inexact Newton iteration mapping},   $N_{F,C,\theta}: B(x_0,\bar{t})  \rightrightarrows  \banacha$, 
\begin{equation} \label{NF}
N_{F,C,\theta}(x):=x+D_{F,C, \theta}(x).
\end{equation}
One can apply a \emph{single} Newton's iteration  on any $x\in  B(x_0,\bar{t})$ to obtain the set $N_{F, C, \theta}(x)$,  which may not be contained  to $B(x_0,\bar{t})$, or even may not be in the domain of $F$. Therefore, this is enough to guarantee the  well-definedness of only one iteration.   To ensure that inexact Newtonian iteration mapping may be repeated indefinitely, we need some additional results.  First, define some subsets of $B(x_0,\bar{t})$ in which, as we shall prove, inexact Newton iteration mappings \eqref{NF} are ``well behaved''.  Define
\begin{equation}\label{eq:ker} 
{K}(t,\varepsilon):=\left\{x\in \banacha\;:\; \Norm{x-x_0}\leq
t,\;\| T_{x_0}^{-1}[-F(x)]\| \leq f(t)+\varepsilon\right\},
\end{equation} 
and
\begin{equation} \label{eq:kt}  
\mathcal{K}:=\bigcup_{(t,\varepsilon)\in {{\mathcal{A}}}}
  K(t,\varepsilon).
\end{equation}
\begin{proposition} \label{pr:syn} 
   Take $0\leq \theta\leq \tilde \theta$ and $N_{F,C,\theta}\in \mathcal{N}_\theta$. Then,    for any $(t,\varepsilon)\in {\mathcal{A}}$ and
  $x\in K(t,\varepsilon)$
   \begin{equation} \label{eq:ifc} 
 \qquad \Norm{y-x}\leq t_+ -t,
  \end{equation} 
where  $y \in N_{F,C,\theta}(x)$ and $t_+$ is the first component of the  function $n_\theta(t,\varepsilon)$ defined in \eqref{eq:nftheta}.   Moreover,
 \begin{equation} \label{eq:inik} 
  N_{F,C,\theta}(K(t,\varepsilon))\subset K(n_\theta(t,\varepsilon)).
  \end{equation} 
  As a consequence,  
   \begin{equation} \label{eq:ist} 
    n_\theta\left( {\mathcal{A}}\right)\subset {\mathcal{A}}, 
    \qquad N_{F,C,\theta}\left( \mathcal{K}\right)\subset  \mathcal{K}.
  \end{equation} 
\end{proposition} 
\begin{proof}
Take  $0\leq \theta $, $(t,\varepsilon)\in {\mathcal{A}}$ and $x\in K(t,\varepsilon)$. Thus, the definitions of the sets  $ {\mathcal{A}} $ in  \eqref{eq:dt},   ${K}(t,\varepsilon)$ in    \eqref{eq:ker}  together with  Lemma~\ref{lm:id.t} imply that 
\begin{equation}\label{eq:ddks} 
\norm{x-x_0}\leq t<\bar t,\qquad  \| T_{x_0}^{-1}[-F(x)]\|\leq f(t)+\varepsilon, \qquad t-(1+\theta)\frac{f(t)+\varepsilon}{f'(t)}< \lambda\leq R.
\end{equation} 
Take  $y\in N_{F,C,\theta}(x)$ and $r$ as in \eqref{def:dfct}.  Using  the third property of convex process in \eqref{eq:dsblm}, we have
$$
T_{x}^{-1}[-F(x)-r]\supseteq T_{x}^{-1}[-F(x)] +T_{x}^{-1}[-r].				
$$
Applying Lemma~\ref{l:incltr}  in each term in the right hand side of last inclusion, one with $w=-r$,  $z=x$ and $v=x_0$, and the other one with $w=-F(x),$ $z=x$ and $v=x_0$,   we obtain
$$
T_{x}^{-1}[-F(x)-r] \supseteq  T_{x}^{-1}F'(x_0)T_{x_0}^{-1}[-F(x)] +T_{x}^{-1}F'(x_0)T_{x_0}^{-1}[-r].
$$
Hence,   taking norm in both sides  of last inclusion and  using  the properties  of the norm   yields 
$$
\left\|T_{x}^{-1}[-F(x)-r]\right\| \leq  \left\|T_{x}^{-1}F'(x_0)\right\|\left\|T_{x_0}^{-1}[-F(x)]\right\| + \left\|T_{x}^{-1}F'(x_0)\right\| \left\|T_{x_0}^{-1}[-r]\right\|.
$$
Considering that $y-x\in D_{F,C,\theta}(x)$, we obtain that $\|y-x\|=\|T_{x}^{-1}[-F(x)-r]\|$. Thus, combining last inequality with Proposition~\ref{wdns} and  the third inequality in \eqref{eq:ddks},  after some manipulation taking into account \eqref{def:dfct},  we have
\begin{equation} \label{eq:occk}
\|y-x\| \leq    -(1+\theta)\frac{f(t)+\varepsilon}{f'(t)},
\end{equation}
which, using definition of $t_{+}$,   is equivalent to \eqref{eq:ifc}.

Since  $\|y-x_0\|\leq \|y-x\| + \|x-x_0\|$,  thus \eqref{eq:occk},     the first and the  last inequality  in \eqref{eq:ddks} give 
\begin{equation} \label{eq:pcck}
 \|y-x_0\|\leq \displaystyle
  t-(1+\theta)\frac{f(t)+\varepsilon}{f'(t)}<\lambda \leq R.
\end{equation}
On the other hand, the linearization error in \eqref{eq:def.er} and  the third property of convex process in \eqref{eq:dsblm} imply 
$$
 T_{x_0}^{-1}[-F(y)]   \supseteq  T_{x_0}^{-1}[-E_F(y,x)]+  T_{x_0}^{-1}[-F(x)-F'(x)(y-x)].
$$
 Thus,  taking norm in both sides  of last inclusion and  using   the triangle inequality we obtain 
$$
   \|T_{x_0}^{-1}[-F(y)\| \leq \Norm{T_{x_0}^{-1}[-E_F(y,x)]}+ \Norm{T_{x_0}^{-1}[-F(x)-F'(x)(y-x)]}.
$$
Since $y\in N_{F,C,\theta}(x)$  we have  $T_{x_0}^{-1}[r]\subset T_{x_0}^{-1}[-F(x)-F'(x)(y-x)]$, where   $r$ satisfies  $F(x)+F'(x)(y-x)+r\in C$ and   \eqref{def:dfct}. Then,  last inequality implies
$$
\|T_{x_0}^{-1}[-F(y)\| \leq \Norm{T_{x_0}^{-1}[-E_F(y,x)]}+   \theta \left\|T_{x_0}^{-1}[-F(x)]\right\|.
$$  
The second term in the right hand side of last inequality is bound by third inequality in \eqref{eq:ddks}. Thus, letting $s=-(1+\theta)(f(t)+\varepsilon)/f'(t)$,  using \eqref{eq:occk},  first and last inequality \eqref{eq:ddks}, we can apply Corollary~\ref{cr:taylor}  to conclude    that
$$
   \|T_{x_0}^{-1}[-F(y)\| \leq e_f\left(t-(1+\theta)\frac{f(t)+\varepsilon}{f'(t)}, t\right)+ \theta( f(t)+\varepsilon).
$$
Therefore, combining  the last inequality with the definition in \eqref{eq:def.ers}, we easily obtain  that
$$
 \| T_{x_0}^{-1}[-F(y)]\| \leq  f \left(t- (1+\theta)  \frac{f(t)+\varepsilon}{f'(t)}\right) +  \varepsilon+2\theta(f(t)+\varepsilon). 
$$
Finally,  \eqref{eq:pcck}, last inequality,  definitions \eqref{eq:nftheta} and \eqref{eq:ker} proof that the inclusion \eqref{eq:inik} holds.
 
The inclusions in \eqref{eq:ist}  are an immediate  consequence  of Lemma~\ref{lm:id.t},  \eqref{eq:inik} and the definitions  in   \eqref{eq:dt}  and  \eqref{eq:kt}. Thus, the proof of the proposition is concluded.
\end{proof}
\subsection{Convergence analysis}
In this section we  will  proof Theorem~\ref{th:knt1}. First we will  show that the sequence generated by inexact Newton method is well behaved with respect to the set defined in \eqref{eq:ker}.
\begin{theorem}
  \label{th:gc.ki.r}
  Take $0\leq\theta\leq\tilde \theta$ and
  $N_{F,C,\theta}\in\mathcal{N}_\theta$.  For any
  $(t_0,\varepsilon_0)\in{\mathcal{A}}$ and $y_0\in K(t_0,\varepsilon_0)$
  the sequences
  \begin{equation}
    \label{eq:seq.02}
   y_{k+1}\in N_{F,C,\theta}(y_k)  ,
\qquad
  ({t}_{k+1},\varepsilon_{k+1})=n_\theta({t}_k,\varepsilon_k) , \qquad k=0,1,\dots,
  \end{equation}
  are well defined,
  \begin{equation}
    \label{eq:in.02} 
  y_k\in K({t}_k,\varepsilon_k)   ,
    \qquad({t}_k,\varepsilon_k)\in {\mathcal{A}} \qquad k=0,1,\dots,
  \end{equation}
  the sequence $\{{t}_k\}$ is strictly increasing and converges to some
  ${\tilde  t}\in (0, \lambda ]$,
  the sequence $\{\varepsilon_k\}$ is non-decreasing and converges to
  some $\tilde \varepsilon\in [0, \kappa\lambda ]$,
  \begin{equation}  \label{eq:bound.02} 
    \left\|T_{x_0}^{-1}[-F(y_k)]\right\|\leq  f({t}_k)+\varepsilon_k\leq
    \left(\frac{1+\theta^2}{2}\right)^{k}(f({t}_0)+\varepsilon_0),
    \quad k=0,1,\dots, 
  \end{equation} 
 $\{y_k\}$ is contained in $B(x_0,\lambda )$,   converges to a point $x_*\in B[x_0, \lambda]$ such that  $ F(x_*)\in C$,  and  satisfies 
  \begin{equation} \label{eq:conv.02}
    \|y_{k+1}-y_k\|\leq {t}_{k+1}-{t}_k, \qquad \|x_{*}-y_k\|\leq
  {\tilde  t}-{t}_k, \qquad k= 0,1, \ldots \,.
  \end{equation}
  Moreover, if   
\begin{itemize}
  \item[{\bf h5')}] $\lambda<R$,
  \end{itemize}
 then  the sequence $\{y_k \}$ satisfies
\begin{equation} \label{eq:qcyk}
\|y_k-y_{k+1}\|\leq \frac{1+\theta}{1-\theta}\, \left[ \frac{1+\theta}{2}\frac{D^{-}f'(\lambda)}{|f'(\lambda)|}\|y_k-y_{k-1}\|
+\theta\,\frac{2+f'(\lambda)}{|f'(\lambda)|}\right]\|y_k-y_{k-1}\| , \qquad k= 0,1,\ldots \,.
\end{equation}
If, additionally, $0\leq \theta <-2(\kappa+1)+\sqrt{4(\kappa+1)^2+\kappa(4+\kappa)}/(4+\kappa)$ then  $\{y_k \}$ converges $Q$-linearly as  follows
  \begin{equation} \label{eq:qlyk}
  \limsup_{k \to \infty}\frac{\norm{x_*-y_{k+1}}}{\norm{x_*-y_k}}\leq \frac{1+\theta}{1-\theta}\, \left[\frac{1+\theta}{2}+\frac{2\theta}{\kappa} \right], \qquad k= 0,1, \ldots \,.
  \end{equation}
\end{theorem}
\begin{proof}
 Since $0\leq\theta\leq\tilde \theta$, $(t_0,\varepsilon_0)\in{\mathcal{A}}$ and $y_0\in K(t_0,\varepsilon_0)$, the  well definition of the sequences $\{(t_k,\varepsilon_k)\}$ and $\{y_k\}$,  as defined in \eqref{eq:seq.02},  follow from the two last inclusions \eqref{eq:ist} in Proposition~\ref{pr:syn}.
  Moreover, since \eqref{eq:in.02} holds for $k=0$, using the first   inclusion in Proposition~\ref{pr:syn}, first inclusion in \eqref{eq:ist} and induction on $k$, we   conclude that \eqref{eq:in.02} holds for all $k$.  The first inequality in \eqref{eq:conv.02}  follows from \eqref{eq:ifc} in Proposition~\ref{pr:syn}, \eqref{eq:seq.02} and  \eqref{eq:in.02},  while the first inequality in \eqref{eq:bound.02} follows from \eqref{eq:in.02} and the definition of $K(t,\varepsilon)$ in \eqref{eq:ker}.  

  The definition of ${\mathcal{A}}$ in \eqref{eq:dt}  implies  $ {\mathcal{A}}\subset [0,\lambda)\times [0,\kappa\lambda)$.   Therefore, using \eqref{eq:in.02} and the definition of  $K(t,\varepsilon)$ we have  
  \[
  t_k\in [0,\lambda),\quad \varepsilon_k\in [0,\kappa\lambda),
  \quad y_k\in B(x_0,\lambda),\qquad k=0,1,\dots .
  \]
  Using \eqref{eq:dt} and Lemma~\ref{lm:id.t} we conclude that
  $\{t_k\}$ is strictly increasing, $\{\varepsilon_k\}$ is
  non-decreasing and the second equality in~\eqref{eq:bound.02} holds
  for all $k$. Therefore, in view of the first two above inclusions,
  $\{t_k\}$ and $\{\varepsilon_k\}$ converge, respectively, to some
  $\tilde t\in (0,\lambda]$ and $\tilde \varepsilon\in
  [0,\kappa\lambda]$. The convergence of $\{t_k\}$ to $\tilde t$, together with the
  first inequality in \eqref{eq:conv.02} and the inclusion $y_k\in
  B(x_0,\lambda)$ implies that $y_k$ converges to some $x_*\in
  B[x_0,\lambda]$ and that the second inequality on \eqref{eq:conv.02}
  holds for all $k$.  Moreover,  taking  the limit  in \eqref{eq:bound.02},  as   $k$ goes to $+\infty$,  we  conclude that   
  $$
  \lim_{k\to +\infty}\left\|T_{x_0}^{-1}[-F(y_k)]\right\|=0.
 $$ 
 Thus,  there exists $\{d_{k}\}\subset \banacha$ such that $ d_{k}\in T_{x_0}^{-1}[-F(y_k)]$,  for all $k=0,1, \ldots$,  with  $\lim_{k\to +\infty} d_k=0$. Since $ d_{k}\in T_{x_0}^{-1}[-F(y_k)]$,  for all $k=0,1, \ldots$, the definition~\ref{ro3} implies that  $F'(x_0)d_k+F(y_k) \in C$ ,  for all $k=0,1, \ldots$. Hence, letting $k$ goes to $+\infty$ in last  inclusion and taking into account that $C$ is closed and $\{y_k\}$ converges to  $x_*$, we conclude that $F(x_*)\in C$. 
	
We are going to prove \eqref{eq:qcyk}. Since $ y_{k+1}\in N_{F,C,\theta}(y_k),$ for k= 0,1, \ldots \,,  we have
\begin{equation}\label{eq:argmin}
\|y_{k+1}-y_k\|= \| T_{y_k}^{-1}[-F(y_k)-r_k]\|, \qquad       \max_{w\in \{-r_k,~ r_k \}} \left\|T_{x_0}^{-1}w\right\| \leq \theta \left\|T_{x_0}^{-1}[-F(y_k)]\right\|.
\end{equation}
The third property  in \eqref{eq:dsblm} implies $T_{y_k}^{-1}[-F(y_k)-r_k] \supseteq T_{y_k}^{-1}[-F(y_k)] +T_{y_k}^{-1}[-r_k]$. Then applying twice Lemma~\ref{l:incltr},  one with   $z=y_k$, $v=x_0$ and  $w= -F(y_k)$  and,  the other one,  with  $z=y_k$, $v=x_0$ and  $w=-r_k$, we obtain that
\begin{equation*}
T_{y_k}^{-1}[-F(y_k)-r_k] \supseteq T_{y_k}^{-1}F'(x_0)T_{x_0}^{-1}[-F(y_k)] +T_{y_k}^{-1}F'(x_0)T_{x_0}^{-1}[-r_k].
\end{equation*}
Combining last inclusion with \eqref{eq:argmin} and properties of the norm  we conclude, after some algebra,  that
\begin{equation}\label{in:cv}
\|y_{k+1}-y_k\| \leq  (1+\theta) \left\|T_{y_k}^{-1}F'(x_0)\right\| \left\|T_{x_0}^{-1}[-F(y_k)]\right\|.
\end{equation}
Using \eqref{eq:def.er}, the third property  in \eqref{eq:dsblm} and triangular inequality, after some  manipulation,  we have
\begin{equation}\label{in:cv1}
  \left\|T_{x_0}^{-1}[-F(y_k)]\right\| \leq \left\|T_{x_0}^{-1}[-E_{F}(y_k,y_{k-1})]\right\| + \left\|T_{x_0}^{-1}[-F(y_{k-1})-F'(y_{k-1})(y_k-y_{k-1})]\right\|.
\end{equation}
On the other hand, because $y_k\in N_{F,C,\theta}(y_{k-1})$  we have  $T_{x_0}^{-1}[r_{k-1}]\subset T_{x_0}^{-1}[-F(y_{k-1})-F'(y_{k-1})(y_k-y_{k-1})]$, where   $r_{k-1}$ satisfies  
$$
\left\|T_{x_0}^{-1}r_{k-1}\right\| \leq \theta \left\|T_{x_0}^{-1}[-F(y_{k-1})]\right\|.
$$ 
Therefore, we have
\begin{equation} \label{eq:bmth}
\left\|T_{x_0}^{-1}[-F(y_{k-1})-F'(y_{k-1})(y_k-y_{k-1})]\right\| \leq   \theta \left\|T_{x_0}^{-1}[-F(y_{k-1})]\right\|, 
\end{equation}
which combined  with    the inequalities in \eqref{in:cv} and \eqref{in:cv1} yields
\begin{equation}\label{in:cv2}
\|y_{k+1}-y_k\| \leq  (1+\theta)\left\|T_{y_k}^{-1}F'(x_0)\right\|\Big[\left\|T_{x_0}^{-1}[-E_{F}(y_k,y_{k-1})]\right\| + \theta\left\|T_{x_0}^{-1}[-F(y_{k-1})] \right\|\Big].
\end{equation}
Using again \eqref{eq:def.er}, the third property  in \eqref{eq:dsblm}  and triangular inequality,  we obtain after some algebra that
$$
\left\|T_{x_0}^{-1}[-F(y_{k-1})\right\| \leq \left\|T_{x_0}^{-1}E_{F}(y_k,y_{k-1})\right\| + \left\|T_{x_0}^{-1} [-F(y_k)]\right\| + \left\|T_{x_0}^{-1}F'(y_{y-1})(y_k-y_{k-1})\right\|.
$$
Combining  the  last inequality with the inequalities in  \eqref{in:cv1} and  \eqref{eq:bmth}  we conclude that
$$
\|T_{x_0}^{-1}[-F(y_{k-1})\| \leq \frac{1}{1-\theta}\left[\|T_{x_0}^{-1}[E_{F}(y_k,y_{k-1})]\| + \|T_{x_0}^{-1}[-E_{F}(y_k,y_{k-1})]\| + \|T_{x_0}^{-1}F'(y_{y-1})(y_k-y_{k-1})\| \right].
$$
Inequality in  \eqref{in:cv2} combined with last inequality becomes
\begin{multline*}
\|y_{k+1}-y_k\| \leq  \frac{1+\theta}{1-\theta} \left\|T_{y_k}^{-1}F'(x_0)\right\| \Big[\left\|T_{x_0}^{-1}[-E_{F}(y_k,y_{k-1})]\right\| + \\  \theta   \left( \left\|T_{x_0}^{-1}[E_{F}(y_k,y_{k-1})]\right\| + \left\|T_{x_0}^{-1}F'(y_{k-1})(y_k-y_{k-1})\right\| \right)\Big].
\end{multline*}
Therefore, combining last inequality with  Proposition~\ref{wdns}, Lemma~\ref{l:bd} and    Corollary~\ref{cr:taylor} with $x=y_{k-1}$, $y=y_{k}$,  $s=t_{k}-t_{k-1}$ and  $t=t_{k-1}$,    we have
\begin{equation} \label{eq:crf}
\|y_k-y_{k+1}\|\leq \frac{1+\theta}{1-\theta}\frac{1}{|f'(t_{k})|}\left[ \frac{1+\theta}{2} \frac{f'(t_{k})-f'(t_{k-1})}{t_{k}-t_{k-1}}\|y_{k-1}-y_k\|
+\theta[2+f'(t_{k-1})]\right]\|y_{k-1}-y_k\|, 
\end{equation}
for $k= 0,1, \ldots \,.$ Since $\|y_{k-1}-y_k\|\leq t_k-t_{k-1}$, see \eqref{eq:conv.02},  $f'<-\kappa<0$ in $[0, \lambda) $,  \eqref{eq:qlyk} follows  from last inequality. Using  {\bf h5' } and  Theorem 4.1.1 on p. 21 of \cite{HiriartBaptiste} and taking into account that $|f'|$ is  decreasing  in $[0, \lambda]$, $f'$ is increasing in  $[0, \lambda]$ and $\{t_k\}\subset [0, \lambda]$,  we obtain  that  \eqref{eq:qcyk}   follows from above inequality.

For concluding the proof, it remains to prove that $\{y_k \}$ converges $Q$-linearly as in  \eqref{eq:qlyk}.  First note that  $\|y_{k-1}-y_k\|\leq t_k-t_{k-1}$ and  $f'(t_{k-1})\leq f'(t_k)<0$. Thus,   we conclude from \eqref{eq:crf} that
\begin{equation} \label{eq:crf}
\|y_k-y_{k+1}\|\leq \frac{1+\theta}{1-\theta}\left[ \frac{1+\theta}{2} +\frac{2\theta}{\kappa}\right]\|y_{k-1}-y_k\|, \qquad k=0,1, \ldots.
\end{equation}
which,  from Proposition 2  of \cite{FGO13},  implies that   \eqref{eq:qlyk} holds. Since $0\leq \theta <-2(\kappa+1)+\sqrt{4(\kappa+1)^2+\kappa(4+\kappa)}/(4+\kappa)$,   the quantity in the right hand side of  \eqref{eq:qlyk} is less than one. Hence  $\{y_k \}$ converges $Q$-linearly, which conclude the proof. 
\end{proof}
\begin{proposition}  \label{pr:ar2}
Let $R>0$ and $ f:[0,R)\to \mathbb{R}$ a continuously differentiable function.   Suppose that $x_0 \in \Omega$,   $f$ is a majorant function for $F$ at  $x_0$ and satisfies {\bf h4}. If  $0\leq \rho<  \beta/2$, then for any $z_0\in  B(x_0,  \rho)$ the scalar  function    $g:[0,R-\rho)\to \mathbb{R}$, defined by
\begin{equation} \label{eq:maj.01}
  g(t):=\frac{-1}{f'(\rho)}[f(t+\rho)+2\rho], 
\end{equation}
is a majorant function for $F$ at $z_0$ and also satisfies condition {\bf h4}.
\end{proposition}
\begin{proof}
To prove see Proposition~17 of \cite{Ferreira2015}.
\end{proof}
\begin{proof}[{\bf Proof of Theorem~\ref{th:knt1}}]
First we will prove Theorem~\ref{th:knt1} with $\rho=0$ and $z_0=x_0$. Note
that, from the definition in \eqref{eq:dktt.0}, we have
\[
 \kappa_0=\kappa,\quad \lambda_0=\lambda,\quad \tilde\theta_0=\tilde\theta.
\]
The assumption \eqref{KH} implies that $x_0\in K(0,0)$. Since $(t_0,\varepsilon_0)=(0,0)\in {\mathcal{A}}$  and $ y_0=x_0\in K(0,0)$,  we apply  Theorem~\ref{th:gc.ki.r}  with   $z_k=y_k$, for $k=0,1, \ldots$,   to  conclude that Theorem~\ref{th:knt1} holds for $\rho=0$ and $z_0=x_0$. 

We are going to prove the general case.  From   Proposition~\ref{pr:maj.f} we have $\rho<\bar t$, which implies  that $\|z_0-x_0\|<\rho<\bar t$. Thus, we can apply Proposition~\ref{wdns}  to obtain
\begin{equation} \label{eq:maj.02}
  \Norm{T_{z_0}^{-1}F'(x_0)}\leq \frac{-1}{f'(\rho)}.
\end{equation}
Moreover,  the point $z_0$ satisfies the Robinson's condition, namely,
$$
\mbox{rge\,}T_{z_0}=\banachb.
$$
Then, using  Lemma~\ref{l:incltr}, property of the norm,   \eqref{eq:maj.02} and Proposition~\ref{pr:new.01} with  $y=z_0$ we have
\begin{align*}
  \Norm{T_{z_0}^{-1}[-F(z_0)]}\leq &  \Norm{T_{z_0}^{-1}F'(x_0)}
  \Norm{T_{x_0}^{-1}[-F(z_0)]}\\
  \leq & \frac{-1}{f'(\rho)} [f(\Norm{z_0-x_0}) +2\Norm{z_0-x_0}].
\end{align*}
Since $f'\geq -1$, the function $t\mapsto f(t)+2t$ is (strictly)
increasing.  Thus,  combining this fact with the last inequality, the inequality $\|z_0-x_0\|<\rho$ and \eqref{eq:maj.01} we conclude that
\[
\Norm{T_{z_0}^{-1}[-F'(z_0)]}\leq g(0).
\]
Proposition~\ref{pr:ar2} implies that  $g$, defined in \eqref{eq:maj.01},  is a majorant function for $F$ at point $z_0$ and also satisfies condition {\bf h4}. Moreover,   \eqref{eq:maj.01}  and $\kappa_\rho$, $\lambda_\rho$  and $\tilde\theta_\rho$ as defined in \eqref{eq:dktt}  imply
\[
  \kappa_\rho=\sup_{0<t<R-\rho}\frac{-g(t)}{t},\qquad
   \lambda_\rho =\sup \{t\in [0,R-\rho): \kappa_\rho+g'(t)<0\},
  \qquad \tilde\theta_\rho=\frac{\kappa_\rho}{2-\kappa_\rho},
\]
which are the same as \eqref{eq:dktt} with $g$ instead of $f$, then  we can  apply Theorem~\ref{th:gc.ki.r} for $F$ and the
majorant function $g$ at point $z_0$ and $\rho=0$, to concluding  that the sequence $\{z_k\}$ is well defined, remains in $B(z_0,\lambda_\rho)$, satisfies (\ref{FC}) and converges to some $x_*\in B[z_0, \lambda_\rho]$ with $F(x_*)\in C$.  Furthermore, since 
$$
g'(t)=f'(t+\rho)/|f'(\rho)|, \qquad D^{-}g'(t)=D^{-}f'(t+\rho)/|f'(\rho)|, \qquad t \in [0,R-\rho), 
$$
 after some algebra, we conclude that inequalities  \eqref{eq:slc} and \eqref{eq:lc} also hold. Therefore, the proof of theorem is concluded. 
\end{proof}
\section{Special cases} \label{sec:scinmer}

In this section  we  will use Theorem~\ref{th:knt1} to analyze    the convergence of the  inexact Newton's  method  for cone inclusion problems under  affine invariant   Lipschitz condition and  in the setting of Smale's $\alpha$-theory.    Up to our knowledge, this is the first time that the inexact Newton method for cone inclusion problems with a relative error tolerance under Lipschitz's condition and Smale's condition are analyzed. 
\subsection{ Under  affine invariant   Lipschitz condition}
In this section we present   the convergence analysis  of the  inexact Newton's  method  for cone inclusion problems under   affine invariant   Lipschitz condition. Let $\banacha$, $\banachb$ be Banach spaces,  $\banacha$ reflexive, $\Omega\subseteq \banacha$ an open set,  $x_0 \in \Omega$ and  $L>0$.  A  continuously Fr\'echet differentiable function $F:{\Omega}\to \banachb$  satisfies the  {\it affine invariant   Lipschitz condition} with constant $L$ at $x_0$, if  $B(x_0, 1/L)\subset \Omega$  and
  \[ \left\|T_{x_0}^{-1}\left[F'(y)-F'(x)\right]\right\| \leq L
  \|x-y\|,\qquad x,\, y\in B(x_0, 1/L).
  \]

\begin{theorem}\label{th:kngerl}
Let   $C\subset \banachb $ a nonempty closed convex cone.  Suppose that $x_0 \in \Omega$ and $F$ satisfies  the Robinson's and the   affine invariant   Lipschitz condition with constant $L>0$ at  $x_0$ and  
  \[ \|T_{x_0}^{-1}F(x_0)\|\leq b,\qquad
  0\leq \theta\leq (1-\sqrt{2bL})/(1+\sqrt{2bL}).
  \] 
  Then,   $\{x_k\}$ generated  by the inexact Newton method for solving $F(x)\in C$ with starting point
  $x_0$ and residual relative error tolerance $\theta$:  $ x_{k+1}:={x_k}+d_k,$
  \[
 d_k \in  \argmin_{d\in \banacha}\left\{\|d\| ~:  ~ F(x_k)+F'(x_k)d +r_k \in C \right\}, \quad \qquad  \max_{w\in \{-r_k,~ r_k  \}} \left\|T_{x_0}^{-1}w\right\|\leq \theta \left\|T_{x_0}^{-1}[-F(x_k)]\right\|, 
  \]
  for all $k=0,1,...$, is well defined,  for any particular choice of each $d_k$,  $\|T_{x_0}^{-1}[-F(x_k)]\|\leq[(1+\theta^2)/2]^{k}b,$   for all $k=0,1,...$,  $\{x_k\}$ is contained in
  $B(x_0, \lambda )$,  converges to a point  $x_*\in B[x_0, \lambda]$,  where $  \lambda :=\sqrt{2bL}/L.$ Moreover,   $\{x_k \}$ satisfies
\[
\|x_k-x_{k+1}\|\leq \frac{1+\theta}{1-\theta}  \left[ \frac{1+\theta}{2}\frac{L}{1-\sqrt{2bL}}\,\|x_{k-1}-x_k\|
+\theta\,\frac{1+\sqrt{2bL}}{1-\sqrt{2bL}}\right]\|x_{k-1}-x_k\|,  \qquad k= 0,1, \ldots .
\]
  If, additionally, $0\leq \theta <\left(-2(2-\sqrt{2bL})+\sqrt{10bL -14\sqrt{2bL}+21}\right)\big/(5-\sqrt{2bL})$ then  $\{x_k \}$ converges $Q$-linearly as  follows
  \[
  \limsup_{k \to \infty}\frac{\norm{x_*-x_{k+1}}}{\norm{x_*-x_k}}\leq  \frac{1+\theta}{1-\theta} \left[\frac{1+\theta}{2}+\frac{2\theta}{1-\sqrt{2bL}} \right], \qquad k= 0,1,
  \ldots \,.
  \]
\end{theorem}

\begin{proof} 
Take $\tilde{\theta}=(1-\sqrt{2bL})/(1+\sqrt{2bL})$. Since   $f:[0,1/L)\to \mathbb{R},$  defined by $f(t):=(L/2)t^2-t+b,$ is a majorant function for $F$ at point $x_0$, all result follow from Theorem~\ref{th:knt1}, applied to this particular context.
\end{proof}
\begin{remark}
In Theorem~\ref{th:kngerl},  if $\theta=0$ and $C=\{0\}$ then  we obtain,  \cite[Theorem 18]{Ferreira2015}  for the  exact Newton method and  \cite[Theorem 6.3]{FerreiraSvaiter2012} for the inexact Newton method,  respectively.
\end{remark}
\subsection{ Under  affine invariant  Smale's condition}
In this section we present    the convergence analysis  of the  inexact Newton's  method  for cone inclusion problems under  affine invariant   Smale's condition.

 Let $\banacha$ and $\banachb$ be Banach spaces, $\Omega \subseteq \banacha$ and   $x_0 \in \Omega$.    A  continuous  function  $F:{\Omega}\to \banachb$  and  analytic in $int(\Omega)$  satisfies the  {\it affine invariant   Smale's condition} with constant $\gamma$ at $x_0$, if  $B(x_0, 1/\gamma)\subset \Omega$  and
  \[ 
   \gamma := \sup _{ n > 1 }\left\| \frac {T_{x_0}^{-1}F^{(n)}(x_0)}{n!}\right\|^{1/(n-1)} <+\infty.
  \]
\begin{theorem} \label{th:kngesrs}

Let   $C\subset \banachb $ a nonempty closed convex cone.  Suppose that $x_0 \in \Omega$ and $F$ satisfies  the Robinson's and the   affine invariant   Smale's condition  with constant $\gamma$ at  $x_0$ and   there exists  $b>0$ such that
  \[
  \|T_{x_0}^{-1}[-F(x_0)]\|\leq b,\qquad b\gamma<3-2\sqrt{2}, \qquad 0\leq\theta\leq[1-2\sqrt{\gamma b}-\gamma b]/[1+2\sqrt{\gamma b}+\gamma b].
  \]
  Then,  $\{x_k\}$ generated  by the inexact Newton method for solving
  $F(x)\in C$ with starting point $x_0$ and residual relative error
  tolerance $\theta$: $x_{k+1}={x_k}+d_k,$
$$
 d_k \in  \emph{argmin}\left\{\|d\| ~:~ d\in \banacha,   ~ F(x_k)+F'(x_k)d +r_k \in C \right\}, \quad \qquad  \max_{w\in \{-r_k,~ r_k  \}} \left\|T_{x_0}^{-1}w\right\|\leq \theta \left\|T_{x_0}^{-1}[-F(x_k)]\right\|, 
$$
 for all $k=0,1,...$, is well defined,  for any particular choice of each $d_k$, $\|T_{x_0}^{-1}[-F(x_k)]\|\leq[(1+\theta^2)/2]^{k}b,$  for all $k=0,1,...$,  $\{x_k\}$  is contained in $B(x_0, \lambda )$ and  converges to a point $x_* \in  B[x_0, \lambda ] $ such that  $F(x_*)\in C$, where $ \lambda:=b/[\sqrt{\gamma b}+\gamma b].$ Moreover, letting $f:[0,1/\gamma) \to \mathbb{R}$ be  defined by   $ f(t)=t/(1-\gamma t)-2t+b,$ the sequence $\{x_k \}$ satisfies
\[
\|x_k-x_{k+1}\|\leq \frac{1+\theta}{1-\theta} \left[ \frac{1+\theta}{2}\frac{D^{-}f'(\lambda)}{|f'(\lambda)|}\|x_{k-1}-x_k\| +\theta\,\frac{2+f'(\lambda)}{|f'(\lambda)|}\right]\|x_{k-1}-x_k\|,   \quad  \qquad k= 0,1, \ldots .
\]
  If, additionally, $0\leq \theta <\left(-2(2-2\sqrt{\gamma b}-\gamma b)+ \sqrt{5\gamma^2 b^2 -44\sqrt{\gamma b}+ 20\gamma b \sqrt{\gamma b} -2\gamma b +21}\right) \big/(5-2\sqrt{\gamma b}-\gamma b)$ then  $\{x_k \}$ converges $Q$-linearly as  follows
  \[
 \limsup_{k \to \infty}\frac{\norm{x_*-x_{k+1}}}{\norm{x_*-x_k}}\leq   \frac{1+\theta}{1-\theta}  \left[\frac{1+\theta}{2} + \frac{2\theta}{1-2\sqrt{\gamma b}-\gamma b} \right], \qquad k= 0,1,
  \ldots \,.
  \]
\end{theorem}

\begin{proof}
Take $\tilde{\theta}=(1-2\sqrt{\gamma b}-\gamma b)/(1+2\sqrt{\gamma b}+\gamma b)$. Use Lemma 20 of \cite{Ferreira2015}  to prove that $f:[0,1/\gamma) \to \mathbb{R}$  defined by   $ f(t)=t/(1-\gamma t)-2t+b,$ is a majorant function for $F$ in $x_0$, see \cite{FerreiraSvaiter2009}. Therefore, 
  all results follow from Theorem~\ref{th:knt1}, applied to this
  particular context.
\end{proof} 
\begin{remark}
In Theorem~\ref{th:kngesrs},  if $\theta=0$ and $C=\{0\}$ then  we obtain,  in the setting of Smale's $\alpha$-theory,   \cite[Theorem 21 ]{Ferreira2015}  for the  exact Newton method and  \cite[Theorem 6.1]{FerreiraSvaiter2012} for the inexact Newton method,  respectively.
\end{remark}
\section{Final remarks} \label{sec:fr}
In this paper we have established a semi-local convergence analysis for inexact Newton's method for cone inclusion problem  under affine invariant majorant condition. Following the same idea of this paper, as future works,  we propose to study the  exact and inexact Newton's method to the problem
\begin{equation}\label{eq:ge}
F(x)+C(x)\ni 0,
\end{equation}
described,  respectively,  by
$$
F(x_k) +F'(x_k)(x_{k+1}-x_k)+C(x_{k+1}) \ni 0 \qquad \;\;k=0,1, \ldots
$$
and
$$
(F(x_k) +F'(x_k)(x_{k+1}-x_k)+C(x_{k+1}))\cap R_{k}(x_k, x_{k+1}) \neq \varnothing, \qquad \;\;k=0,1, \ldots , 
$$ 
where $C: \banacha \times \banacha \rightrightarrows \banachb$ is now a  set-value mapping and   $R_{k}: \banacha \times \banacha \rightrightarrows \banachb$ is a sequence of set-value mappings with closed graphs. The  problem \eqref{eq:ge} is a generalization to problem \eqref{eq:ipi}, which is called generalized equations, and it has been the subject of many new research, see \cite{Dontchev2015, DontchevRockafellar2009, DontchevRockafellar2010, DontchevRockafellar2013, PietrusJean2013}. Furthermore, it will be interesting   to study these two above methods under a majorant condition.  

\end{document}